\documentclass[11pt,reqno,a4paper]{amsart}

\usepackage{amsmath}
\usepackage{enumerate}
\usepackage{amssymb}
\usepackage{amscd}
\usepackage{amsthm}
\usepackage{amsfonts}
\usepackage{graphicx}
\usepackage{hyperref}

\usepackage{enumitem}

\usepackage{ccfonts,eulervm} 
\usepackage[T1]{fontenc}

\numberwithin{equation}{section}

\newcommand{\SL}{\operatorname{SL}}
\newcommand{\gsl}{\operatorname{\mathfrak{s}\mathfrak{l}}}

\newcommand{\GL}{\operatorname{GL}}
\newcommand{\PSL}{\operatorname{PSL}}
\newcommand{\Aut}{\operatorname{Aut}}

\newcommand{\cP}{\mathcal{P}}

\newcommand{\bC}{\mathbb{C}}

\newcommand{\bQ}{\mathbb{Q}}
\newcommand{\bR}{\mathbb{R}}

\newcommand{\bT}{\mathbb{T}}

\newcommand{\bZ}{\mathbb{Z}}

\newcommand{\ra}{\rightarrow}

\newcommand{\qand}{\quad \textrm{and} \quad}

\def\acts{\curvearrowright}
\newcommand\subsetsim{\mathrel{%
\ooalign{\raise0.2ex\hbox{$\subset$}\cr\hidewidth\raise-0.8ex\hbox{\scalebox{0.9}{$\sim$}}\hidewidth\cr}}}

\DeclareMathOperator{\Ad}{Ad}

\DeclareMathOperator{\trace}{tr}

\theoremstyle{theorem}
\newtheorem{theorem}{Theorem}[section]
\newtheorem{corollary}[theorem]{Corollary}
\newtheorem{proposition}[theorem]{Proposition}

\theoremstyle{definition}
\newtheorem{definition}{Definition}[section]
\newtheorem{remark}[theorem]{Remark}

\begin{document}

\title{Characteristic Polynomial Patterns in difference sets of matrices}

\author{Michael Bj\"orklund}
\address{Department of Mathematics, Chalmers, Gothenburg, Sweden}
\email{micbjo@chalmers.se}
\thanks{}

\author{Alexander Fish}
\address{School of Mathematics and Statistics, University of Sydney, Australia}
\curraddr{}
\email{alexander.fish@sydney.edu.au}
\thanks{}



\date{}

\dedicatory{}

\begin{abstract}
We show that for every subset $E$ of positive density in the set of integer square-matrices with zero traces,
there exists an integer $k \geq 1$ such that the set of characteristic polynomials of matrices in $E-E$ contains the 
set of \emph{all} characteristic polynomials of integer matrices with zero traces and entries divisible by $k$.
Our theorem is derived from results by Benoist-Quint on measure rigidity for actions on homogeneous
spaces.
\end{abstract}

\maketitle

\section{Introduction}
We recall the celebrated Furstenberg-Sark\"ozy Theorem \cite{F81}, \cite{S}. Let $E_o \subset \bZ$ be a set
with
\[
\overline{d}_{\bZ}(E_o) = \varlimsup_{n \ra \infty} \frac{|E_o \cap [1,n]|}{n} > 0,
\]
and let $p \in \bZ[X]$ be a polynomial with $p(0) = 0$. Then, there exists $n \geq 1$ such that
\[
p(n) \in E_o - E_o = \{ x-y \, : \, x,y \in E_o \big\}.
\]
In other words, the difference set of any set of positive density in $\bZ$ contains "polynomial patterns".

In this paper, we establish an analogue of Furstenberg-Sark\"ozy Theorem for difference sets 
of matrices. Let $M_d(\bZ)$ denote the additive group of $d \times d$-integer matrices, and let 
$M_d^0(\bZ)$ denote the subgroup of $M_d(\bZ)$ consisting of matrices with zero trace. For a 
subset $E \subset M_d^{0}(\bZ)$, we define its \emph{upper asymptotic density} by
\[
\overline{d}(E) = \varlimsup_{n \ra \infty} \frac{|E \cap F_n|}{|F_n|},
\]
where $F_n 
= 
\big\{
A = (a_{ij}) \in M^{0}_d(\bZ) \, : \, |a_{ij}| \leq n, \, \enskip \textrm{for all $(i,j) \neq (d,d)$}
\big\}$. \\ 

The main result of this paper can be formulated as follows.
\begin{theorem}
\label{main}
For every integer $d \geq 2$ and $E \subset M_d^{0}(\bZ)$ with $\overline{d}(E) > 0$, there exists an integer $k \geq 1$ such that for every $f \in \bZ[X]$ of the form
\begin{equation}
\label{charpol}
f(X) = X^d + k^2 \cdot a_{d-2} X^{d-2} + \ldots + k^d \cdot a_o, \quad \textrm{where $a_o,\ldots, a_{d-2} \in \bZ$},
\end{equation}
there exists a matrix $A \in E - E$ such that $f$ is the characteristic polynomial of $A$.
\end{theorem}

By evaluating the characteristic polynomials for elements in $E-E$ at $X = 0$, we get the following corollary.

\begin{corollary}
For every integer $d \geq 2$ and $E \subset M_d^{0}(\bZ)$ with $\overline{d}(E) > 0$, the set
\[
D = \big\{ \det(A) \, : \, A \in E-E \big\} \subset \bZ
\]
contains a non-trivial subgroup. 
\end{corollary}

\begin{remark}
We note that for every $k \geq 1$, the subgroup $E = k \cdot M_d^{0}(\bZ) \subset M_d^0(\bZ)$ has 
positive density and all characteristic polynomials of elements $A \in E-E$ have the form  \eqref{charpol}. 
Hence, in this case, our theorem is sharp.
\end{remark}
It is worth pointing out that there are sets $E \subset M_d(\bZ)$ with 
\[
\overline{d}(E) = \varlimsup_{n \ra \infty} \frac{|E \cap G_n|}{|G_n|} > 0,
\]
where $G_n = 
\big\{
A = (a_{ij}) \in M_d(\bZ) \, : \, |a_{ij}| \leq n, \, \enskip \textrm{for all $(i,j)$}
\big\},
$
such that the set 
\[
T = \big\{ \trace(A-A') \, : \, A, A' \in E \big\} \subset \bZ
\] 
does not contain a non-trivial subgroup. In other words, there exists a subset $E \subset M_d(\bZ)$
of positive density with the property that the set of characteristic polynomials of elements in the difference set $E-E$ does not
contain the set
\[
C_k = \big\{ f \in \bZ[X] \, : \, f(X) = \det(X \cdot I - A), \enskip \textrm{for $A \in k \cdot M_d(\bZ)$}\big\},
\]
for any integer $k \geq 1$. Indeed, let $\alpha \in \bR$ be an irrational number and denote by
$I \subset \bR/\bZ$ an open interval such that the closure of $I-I \subset \bR/\bZ$ is a proper subset. Define 
\[
E_o = \big\{ n \in \bZ \, : \, n\alpha \enskip \textrm{mod} \: 1 \in I \big\} \subset \bZ
\]
and note that $E_o - E_o \subset \big\{ n \, : \, n \alpha \enskip \textrm{mod} \: \,1 \in I-I \big\} \subset \bZ$ does not contain a non-trivial subgroup. The set
\[
E = \big\{ A \in M_d(\bZ) \, : \, \trace(A) \in E_o \big\} \subset M_d(\bZ)
\]
satisfies $\overline{d}(E) > 0$, and $T = E_o - E_o$. \\

As another application of our main theorem, we prove a "sum-product" analogue of Bogolyubov's Theorem (see e.g. 
Theorem 7.2 in \cite{R}). 

\begin{corollary}
For every $E_o \subset \bZ$ with $\overline{d}_{\bZ}(E_o) > 0$, the set
\[
D = \big\{ xy-z^2 \, : \, x,y,z \in E_o-E_o \big\} \subset \bZ
\]
contains a non-trivial subgroup of $\bZ$.
\end{corollary}

\begin{proof}
Fix a set $E_o \subset \bZ$ with $\overline{d}_{\bZ}(E_o) > 0$ and define the set
\[
E = \Big\{ 
\left(
\begin{array}{cc}
a & -b \\
c & -a
\end{array}
\right)
\, : \, 
a,b,c \in E_o \Big\} \subset M_2^0(\bZ).
\]
One can readily check that $\overline{d}(E) > 0$, and thus, by Theorem \ref{main}, there exists an integer $k \geq 1$
such that for every $f \in \bZ[X]$ of the form
\[
f(X) = X^2 + k^2 \cdot a_{0}, \quad \textrm{where $a_o \in \bZ$},
\]
there exists an element $A \in E-E$ with $f(X) = \det(X \cdot I-A)$. In particular, given any 
integer $a_o$, we can find a matrix
\[
A
=
\left(
\begin{array}{cc}
z & -y \\
x & -z
\end{array}
\right) 
\]
with $x,y,z \in E_o-E_o$, whose characteristic polynomial has the form $f(X) = X^2 + k^2 \cdot a_o$. Hence,
\[
f(0) = \det(-A) = xy - z^2 = k^2 \cdot a_o,
\]
which shows that $k^2 \cdot \bZ \subset D$.
\end{proof}

We note that Theorem \ref{main} is an immediate consequence of the following theorem.
\begin{theorem}
\label{prop1}
For every integer $d \geq 2$ and $E \subset M_{d}^0(\bZ)$ with $\overline{d}(E) > 0$, there exists an integer 
$k \geq 1$ such that for every $A \in k \cdot M_d^0(\bZ)$, we have
\[
\overline{d}\big(E \cap (E - gAg^{-1})\big) > 0, \quad \textrm{for some $g \in \SL_d(\bZ)$}.
\]
\end{theorem}

\begin{proof}[Proof of Theorem \ref{main} using Theorem \ref{prop1}]
Fix $d \geq 2$ and $a_o,\ldots, a_{d-2} \in \bZ$ and pick an element $A_o \in M_d^0(\bZ)$ whose 
characteristic polynomial $f_o$ has the form
\[
f_o(X) = X^d + a_{d-2} X^{d-2} + \ldots + a_o.
\]
Fix a set $E \subset M_d^0(\bZ)$ with $\overline{d}(E) > 0$ and use Theorem \ref{prop1} to find an integer 
$k \geq 1$ such that, for every $A \in k \cdot M_d^0(\bZ)$, we have
\[
E \cap (E-gAg^{-1}) \neq \emptyset, \quad \textrm{for some $g \in \SL_d(\bZ)$}.
\]
In particular, we can take $A = k \cdot A_o$, and we conclude that $k \cdot gA_o g^{-1} \in E - E$. Since the
characteristic polynomial $f$ of $k \cdot A_o$ (and $k \cdot gA_o g^{-1}$) equals 
\[
f(X) = X^d + k^2 \cdot a_{d-2} X^{d-2} + \ldots + k^d \cdot a_o,
\]
and $a_o,\ldots, a_{d-2}$ are arbitrary integers, we are done.
\end{proof}

We now say a few words about the strategy of the proof of Theorem \ref{prop1}. The basic steps can be
summarized as follows:
\begin{itemize}
\item In Section \ref{section1} we reduce the theorem to a problem concerning recurrence of $\SL_d(\bZ)$-conjugation
orbits in $M_d^{0}(\bZ)$.
\item In Section \ref{section2} we show that this kind of recurrence can be linked to the behavior of random walks on 
$\SL_d(\bZ)$ acting on the dual group of $M_d^{0}(\bZ)$. 
\item In Section \ref{section3} and Section \ref{section4} we use the work on measure rigidity by Benoist-Quint \cite{BQ} to 
establish the necessary recurrence.
\end{itemize}

\section{Proof of Theorem \ref{prop1}}
\label{section1}
Let $H_d = M_d^{0}(\bZ)$ and recall that the \emph{dual} $T_d$ of $H_d$ is defined 
as the multiplicative group of all homomorphisms $\chi : H_d \ra \bT$, where 
$\bT = \{ z \in \bC^{*} \, : \, |z| = 1 \}$. We note that $T_d$ is a compact metrizable abelian
group and that we have a natural isomorphism 
$M_d^0(\bR)/M_d^{0}(\bZ) \ra T_d$ given by $\Theta \mapsto \chi_\Theta$, where 
\[
\chi_\Theta(A) = e^{2\pi i \trace(\Theta^t A)}, \quad \textrm{for $A \in H_d$}.
\]
We denote by $1$ the trivial character on $T_d$ (the one corresponding to $\Theta = 0$), and
we let $\cP(T_d)$ denote the space of Borel probability measures on $T_d$. \\

Given $A \in H_d$, we define $\phi_A(\chi) = \chi(A)$ for $\chi \in T_d$, and given a Borel probability 
measure $\eta$ on $T_d$, we define its \emph{Fourier transform} $\widehat{\eta}$ by
\[
\widehat{\eta}(A) = \int_{T_d} \phi_A(\chi) \, d\eta(\chi) = \int_{T_d} \chi(A) \, d\eta(\chi), \quad \textrm{for $A \in H_d$}.
\]
The following proposition implies Theorem \ref{prop1}.
\begin{proposition}
\label{prop2}
For every $d \geq 2$ and $\eta \in \cP(T_d)$ with $\eta(\{1\}) > 0$, there exists $k \geq 1$ such that 
for every $A \in k \cdot M_d^0(\bZ)$, we have
\[
\widehat{\eta}(gAg^{-1}) \neq 0, \quad \textrm{for some $g \in \SL_d(\bZ)$}.
\]
\end{proposition}

\begin{proof}[Proof of Theorem \ref{prop1} using Proposition \ref{prop2}]
By the proof of Furstenberg's Correspondence Principle (see Section 1, \cite{F77}) 
for the countable \emph{abelian} 
group $H_d = M_d^0(\bZ)$, we can find a compact metrizable space $Z$, equipped
with an action of $H_d$ on $Z$ by homeomorphisms, denoted by $(A,z) \mapsto A \cdot z$, a $H$-invariant (not necessarily
ergodic) Borel probability measure $\nu$ on $Z$ and a Borel set $B \subset Z$ with $\nu(B) > 0$
such that
\[
\overline{d}\big(E \cap (E - A)\big) \geq \nu(B \cap A \cdot B), \quad \textrm{for all $A \in H_d$}.
\]
We note that $A \mapsto \nu(B \cap A \cdot B)$ is a positive definite function on $H_d$, and 
thus, by Bochner's Theorem (Theorem 4.18 in \cite{Fo}), we can find a probability measure $\eta$ on the dual 
group $T_d = \widehat{H}_d$, such that
\[
\frac{\nu(B \cap A \cdot B)}{\nu(B)} = \widehat{\eta}(A) = \int_{T_d} \chi(A) \, d\eta(\chi), \quad \textrm{for all $A \in H_d$}. 
\]
Furthermore, by the weak Ergodic Theorem, using the fact that $\nu(B) > 0$, we have $\eta(\{1\}) > 0$. By Proposition \ref{prop2}, we can find an integer $k \geq 1$ such that for every $A \in k \cdot H_d$, we have
\[
 \widehat{\eta}(gAg^{-1}) \neq 0, \quad \textrm{for some $g \in \SL_d(\bZ)$}
\]
and thus, $\nu(B \cap (gAg^{-1}) \cdot B) > 0$, and 
\[
\overline{d}\big(E \cap (E - gAg^{-1})\big) 
\geq 
\nu(B \cap (gAg^{-1}) \cdot B) > 0,
\]
for some $g \in \SL_d(\bZ)$, which finishes the proof.
\end{proof}

\section{Stationary measures and the proof of Proposition \ref{prop2}}
\label{section2}
The main point of this section is to show that it suffices to establish Proposition 
\ref{prop2} for a more restrictive class of Borel probability measures on $\bT_d$. \\

Let $\mu$ be a probability measure on $\SL_d(\bZ)$. We say that $\mu$ is 
\emph{generating} if its support generates $\SL_d(\bZ)$ as a semigroup, and
we say that $\mu$ is \emph{finitely supported} if its support is finite. Given an
integer $n \geq 1$, we define 
\[
\mu^{*n}(g) = \sum_{g_1 \cdots g_n = g} \mu(g_1) \mu(g_2) \cdots  \mu(g_n), \quad
\textrm{for $g \in \SL_d(\bZ)$},
\]
where the sum is taken over all $n$-tuples $(g_1,\ldots,g_n)$ in $\SL_d(\bZ)$ such that 
$g_1 \ldots g_n = g$. Recall that $T_d = \widehat{H}_d$, and $\SL_d(\bZ)$ acts on 
$T_d$ by
\[
\big(g \cdot \chi\big)(A) = \chi(g^{-1} A g), \quad \textrm{for $A \in \SL_d(\bZ)$ and $\chi \in T_d$}.
\]
We note that this induces a weak*-homeomorphic action of $\SL_d(\bZ)$ on the space $\cP(T_d)$ of Borel
probability measures on $T_d$ (which we shall here think of as elements in the dual of the space 
$C(T_d)$ of continuous functions on $T_d$) by
\[
\int_{T_d} \phi(\chi) \, d(g \cdot \eta)(\chi) = \int_{T_d} \phi(g \cdot \chi) \, d\eta(\chi),
\quad \textrm{for $\eta \in \cP(T_d)$ and $\phi \in C(T_d)$}.
\]
Furthermore, we define the Borel probability measure $\mu * \eta$ on $T_d$ by
\[
\int_{T_d} \phi(\chi) \, d(\mu * \eta)(\chi) =  \sum_{g \in \SL_d(\bZ)}
\Big( \int_{T_d} \phi(g \cdot \chi) \, d\eta(\chi) \Big) \cdot \mu(g), \quad \textrm{for $\phi \in C(T_d)$}.
\]
In particular, given $A \in M_d^0(\bZ)$, we let $\phi_A$ denote the character on $T_d$ given by
$\phi_A(\chi) = \chi(A)$ for $\chi \in T_d$, and we note that
\begin{eqnarray*}
\widehat{\mu * \eta}(A) 
&=& 
\int_{T_d} \phi_A(\chi) \, d(\mu * \eta)(\chi) 
= \sum_{g \in \SL_d(\bZ)} \int_{T_d} \chi(g^{-1}Ag) \, d\eta(\chi) \, d\mu(g) \\
&=& \sum_{g \in \SL_d(\bZ)} \widehat{\eta}(g^{-1}Ag) \cdot \mu(g), \quad \textrm{for all $A \in M_d^0(\bZ)$}.
\end{eqnarray*}
We say that a Borel probability measure $\xi$ on $T_d$ is \emph{$\mu$-stationary} if $\mu * \xi = \xi$. It is 
not hard to prove (see e.g. Proposition 3.3, \cite{BF3}) that the set $\cP_\mu(T_d)$ of $\mu$-stationary Borel probability measures
on $T_d$ is never empty, and the measure class of any element $\xi \in \cP_\mu(T_d)$ is invariant under 
the semi-group generated by the support of $\mu$. If $\mu$ is a generating measure on $\SL_d(\bZ)$, we 
say that an element $\xi \in \cP_\mu(T_d)$ is \emph{ergodic} if a $\SL_d(\bZ)$-invariant Borel set in $T_d$
is either $\xi$-null or $\xi$-conull. \\

The following proposition implies Proposition \ref{prop2}.

\begin{proposition}
\label{prop3}
There exists a finitely supported probability measure $\mu$ on $\SL_d(\bZ)$ whose support generates
$\SL_d(\bZ)$ with the property that for every $\xi \in \cP_\mu(T_d)$ with $\xi(\{1\}) > 0$, there exists
an integer $k \geq 1$ such that for every $A \in k \cdot M_d^0(\bZ)$, we have
\[
\widehat{\xi}(g^{-1}Ag) \neq 0, \quad \textrm{for some $g \in \SL_d(\bZ)$}.
\]
\end{proposition}

\begin{proof}[Proof of Proposition \ref{prop2} using Proposition \ref{prop3}]
Pick $\eta \in \cP(T_d)$ with $\eta(\{1\}) > 0$, and write
\[
\eta = \lambda \cdot \delta_1 + (1-\lambda) \cdot \eta_o, \quad \textrm{for some $0 < \lambda \leq 1$}, 
\]
where $\eta_o(\{1\}) = 0$. Since $1$ is fixed by the $\SL_d(\bZ)$-action, we have
\[
\mu^{*n} * \eta = \lambda \cdot \delta_1 + (1-\lambda) \cdot \mu^{*n} * \eta_o, \quad \textrm{for every $n \geq 1$},
\]
and thus
\[
\eta_N = \frac{1}{N} \sum_{n=1}^{N} \mu^{*n} * \eta 
= 
\lambda \cdot \delta_1 + (1-\lambda) \cdot \frac{1}{N} \sum_{n=1}^N \mu^{*n} * \eta_o, \quad \textrm{for every $N \geq 1$}.
\]
Since $\cP(T_d)$ is weak*-compact, we can find a subsequence $(N_j)$ such that $\eta_{N_j}$ converges to a 
probability measure $\xi$ on $T_d$ in the weak*-topology, which must be $\mu$-stationary and satisfy the 
bound $\xi(\{1\}) \geq \lambda > 0$. By Proposition \ref{prop3}, there exists an integer $k \geq 1$ such that 
for every $A \in k \cdot H_d$, we have
\[
\widehat{\xi}(g^{-1}Ag) \neq 0, \quad \textrm{for some $g \in \SL_d(\bZ)$}.
\]
We now claim that for every $A \in k \cdot H_d$, we have 
\[
\widehat{\eta}(g^{-1}Ag) \neq 0, \quad \textrm{for some $g \in \SL_d(\bZ)$}. 
\]
Indeed, suppose that this is not the case, so that $\widehat{\eta}(g^{-1}Ag) = 0$ for all $g \in \SL_d(\bZ)$, and thus
\[
\widehat{\eta}_N(h^{-1}Ah) =  \frac{1}{N} \sum_{n=1}^N \sum_{g \in \SL_d(\bZ)}\widehat{\eta}(g^{-1}h^{-1}Ah g) \cdot \mu^{*n}(g) = 0, 
\]
for all $N \geq 1$ and $h \in \SL_d(\bZ)$. Since $\eta_{N_j} \ra \xi$ in the weak*-topology, we conclude that 
we must have $\widehat{\xi}(h^{-1} A h) = 0$ for all $h \in \SL_d(\bZ)$, which is a contradiction.
\end{proof}

\section{Measure rigidity and the proof of Proposition \ref{prop3}}
\label{section3}

\begin{definition}
\label{nice}
Let $X$ be a compact \emph{abelian} group and $\Gamma < \Aut(X)$. Let $\mu$ be a generating probability measure on $\Gamma$.
We say that the action of $\Gamma$ on $X$ is \emph{$\mu$-nice} if the following conditions are satisfied:
\begin{itemize}
\item Every ergodic and $\mu$-stationary Borel probability measure on $X$ is either the Haar measure $m_X$ or supported on a finite 
$\Gamma$-orbit in $X$.
\item There are only countably many finite $\Gamma$-orbits in $X$, and each element in a finite $\Gamma$-orbit
has finite order.
\end{itemize}
In particular, by the ergodic decomposition for $\mu$-stationary Borel probability measures, see e.g. Proposition 
3.13, \cite{BF3}, if the $\Gamma$-action is $\mu$-nice, then every $\mu$-stationary (not necessarily ergodic) 
Borel probability measure 
$\xi$ on $X$ can be written as
\[
\xi = r \cdot m_X + (1-r) \cdot \sum_{P} q_P \cdot \nu_P, \quad \textrm{for some $0 \leq r \leq 1$},
\]
where $\nu_P$ denotes the counting probability measure on a finite $\Gamma$-orbit $P \subset X$, and 
$q_P$ are non-negative real numbers such that $\sum_P q_P = 1$.
\end{definition}

In order to prove Proposition \ref{prop3}, we shall need the following "measure rigidity" result, which will
be proved in Section \ref{section4} using results by Benoist-Quint \cite{BQ}.

\begin{proposition}
\label{prop4}
For every finitely supported generating probability measure $\mu$ on $\SL_d(\bZ)$, the dual 
action $\SL_d(\bZ) \acts T_d$ is $\mu$-nice.
\end{proposition}

\begin{proof}[Proof of Proposition \ref{prop3} using Proposition \ref{prop4}]
Fix $\xi \in \cP_\mu(T_d)$ with $\xi(\{1\}) = q > 0$ and a finitely supported generating probability measure 
$\mu$ on $\SL_d(\bZ)$. Since the dual action of $\SL_d(\bZ)$ on $T_d$ is $\mu$-nice by Proposition \ref{prop4}, 
we can write $\xi$ as
\[
\xi = q \cdot \delta_1 + r \cdot m_X + (1-r-q) \cdot \sum_{P \neq \{1\}} q_P \cdot \nu_P, 
\]
for some $r \geq 0$ with $0 < r + q \leq 1$, where $\nu_P$ and $q_P$ are as in Definition \ref{nice}, and thus
\[
\widehat{\xi} = q + r \cdot \delta_0 + (1-r-q) \cdot \sum_{P \neq \{1\}} q_P \cdot \widehat{\nu}_P.
\]
If $q + r = 1$, then $\widehat{\xi}(A) \geq q > 0$ for every $A \in H_d$, so we may assume from now on that 
the inequalities $0 < r+q < 1$ hold. Since $(q_P)$ is summable, we can find a finite subset $F$ of the set of finite $\SL_d(\bZ)$-orbits in $T_d$ such that
\[
\sum_{P \notin F} q_P < \frac{q}{1-r-q}.
\]
Since the action is $\mu$-nice, we note that, for each finite $\SL_d(\bZ)$-orbit $P$, every element in $P$ has 
finite order, and thus we can find an integer $n_P$ such that $\chi^{n_P} = 1$ for all $\chi \in P$. Since $F$ is 
finite, we can further find an integer $k$ such that $\chi^k = 1$ for every $\chi \in P$ and for every $P \in F$.
Hence, $\chi(k \cdot A) = 1$ for all $A \in H_d$ and for every $\chi \in P$ and for every $P \in F$, and thus
\[
\widehat{\nu}_P(k \cdot A) = \frac{1}{P} \sum_{\chi \in P} \chi(k \cdot A) = 1, \quad \textrm{for all $A \in H_d$}.
\]
We conclude that
\[
\widehat{\xi}(k \cdot A) = q + (1-r-q) \cdot \sum_{P \in F} q_P + (1-r-q) \cdot \sum_{P \notin F} q_P \cdot \widehat{\nu_P}(k \cdot A),
\]
for every non-zero $A \in H_d$, and thus
\[
\big| \widehat{\xi}(k \cdot A) \big| \geq q - (1-r-q) \cdot \sum_{P \notin F} q_P > 0,
\]
since $|\widehat{\nu}_P(A)| \leq 1$ for every $A \in H_d$, which finishes the proof.
\end{proof}

\section{Proof of Proposition \ref{prop4}}
\label{section4}
Let us briefly recall the setting so far. We have
\[
H_d = M_d^0(\bZ) \qand T_d = \widehat{H}_d \cong M_d^0(\bR) / M_d^0(\bZ)  
\]
and a polynomial homomorphism $\Ad : \SL_d(\bR) \ra \GL(M_d^0(\bR))$ defined by
\[
\Ad(g)A = (g^{t})^{-1}Ag^t, \quad \textrm{for $g \in \SL_d(\bR)$ and $A \in M_d^0(\bR)$},
\]
where $g^t$ denotes the transpose of $g$. \\

We note that $\Ad(g)M_d^0(\bZ) = M_d^0(\bZ)$ for all 
$g \in \SL_d(\bZ)$ and thus we can define a homeomorphic action of the group $\SL_d(\bZ)$ on $M_d^0(\bR) / M_d^0(\bZ)$ by
\[
g \cdot (A+M_d^0(\bZ)) = \Ad(g)A + M_d^{0}(\bZ), \quad \textrm{for $A + M_d^0(\bZ) \in M_d^0(\bR) / M_d^0(\bZ)$}.
\]
We note that this action of $\SL_d(\bZ)$ is isomorphic to the one on $T_d$ via the map $\Theta \mapsto \chi_\Theta$
introduced in Section \ref{section2}. 

We wish to prove that for every finitely supported generating probability measure 
$\mu$ on the group $\Ad(\SL_d(\bZ)) < \Aut(T_d)$, the action on $T_d$ is $\mu$-nice. \\

This is a special case of the following more general setting. Let $V$ be a real finite-dimensional
vector space and suppose that $\rho : \SL_d(\bR) \ra \GL(V)$ is a polynomial homomorphism 
defined over $\bQ$ and set $\Gamma = \rho(\SL_d(\bZ))$. Let $\Lambda < V$ be a
subgroup which is isomorphic to $\bZ^{n}$, where $n = \dim_{\bR}(V)$, so that the quotient 
group $X = V/\Lambda$ is compact. In the setting described above, we have
\[
V = M_d^{0}(\bR) \qand \Lambda = M_d^{0}(\bZ) \qand \rho = \Ad \qand n = d^2-1.
\]
Recall that the action of a subgroup $G < \GL(V)$ is \emph{irreducible} if it does not admit any
non-trivial proper $G$-invariant subspaces, and we say that it is \emph{strongly irreducible} if
the action of any finite-index subgroup of $G$ is irreducible. The following theorem of 
Benoist-Quint (Theorem 1.3, \cite{BQ}) will be the main technical ingredient in the proof of Proposition \ref{prop4}.
\begin{theorem}
\label{BQ}
Let $\mu$ be a finitely supported generating probability measure on $\Gamma$ and suppose that
$\Gamma \acts V$ is strongly irreducible. Then a $\mu$-stationary ergodic probability measure
on $X$ is either the Haar measure on $X$ or the counting probability measure on some finite 
$\Gamma$-orbit in $X$.
\end{theorem}

Given a subset $Y \subset \GL(V)$, we denote by $\overline{Y}^{Z}$ the Zariski closure of $Y$.
The following proposition provides a condition which ensures that $\Gamma$ acts strongly 
irreducibly on $V$.

\begin{proposition}
\label{BQ1}
Suppose that $\overline{\Gamma}^{Z} = G < \GL(V)$ is a Zariski-connected group which acts irreducibly on $V$. Then, $\Gamma$ acts strongly irreducibly on $V$, and for every finite-index subgroup $\Gamma_o < \Gamma$, any non-trivial $\Gamma_o$-invariant subgroup of $\Lambda$ has finite index. 

Furthermore, there are countably many finite $\Gamma$-orbits
in $X$, and for every finite $\Gamma$-orbit $P \subset X$ there exists an integer $n$ such that $\chi^n = 1$
for all $\chi \in P$.
\end{proposition}

\begin{proof}
Suppose that $\Gamma_o$ is a finite-index subgroup of $\Gamma$ and let $U < V$ be a non-trivial $\Gamma_o$-invariant subspace. Since $G$ is connected it must also be equal to the Zariski closure of $\Gamma_o$, and thus $U$
is also fixed by $G$ (since $\rho$ is a polynomial map and being invariant subspace is an algebraic condition). Hence, $U = V$. This shows that $\Gamma$ acts strongly irreducibly. 

Now suppose that $\Lambda_o < \Lambda$ is a non-trivial $\Gamma_o$-invariant subgroup. Since 
$\Lambda$ is assumed to be isomorphic to $\bZ^n$ for some $n$, and every subgroup of a free abelian
group is free, we can find a $\bZ$-basis $e_1,\ldots,e_m$ of $\Lambda_o$, and one readily checks that
the real subspace
\[
U := \bR e_1 \oplus \ldots \oplus \bR e_m < V
\]
is $\Gamma_o$-invariant as well. Since the $\Gamma$-action on $V$ is assumed to be strongly irreducible, 
we can conclude that $U = V$, and thus $m = n$. From this it follows that the subgroup $\Lambda_o$ has finite index in $\Lambda \cong \bZ^n$. 

Now suppose that $P \subset X$ is a finite $\Gamma$-orbit, and pick $\chi_o \in P$. We note that there exists a 
finite-index subgroup $\Gamma_o$ of $\Gamma$ which fixes $\chi_o$, and thus the kernel $\Lambda_o = \ker \chi_o$ is a non-trivial $\Gamma_o$-invariant subgroup of $\Lambda$. Hence, from the previous paragraph, it must have finite
index in $\Lambda$, and thus $\chi_o$ has finite order in $X$. Since there are only countably many finite-index 
subgroups of $\Lambda \cong \bZ^n$, we conclude that there are only countably many choices of elements $\chi_o$ 
in $X$ which belong to a finite $\Gamma$-orbit, and thus there are at most countably many finite $\Gamma$-orbits in $X$. 
\end{proof}

\begin{corollary}
\label{niceness}
Let $\mu$ be a finitely supported generating probability measure on $\Gamma$ and suppose that
$\Gamma \acts V$ is strongly irreducible. Then the $\Gamma$-action on $X$ is $\mu$-nice.
\end{corollary}

\begin{proof}
By Theorem \ref{BQ}, a $\mu$-stationary and \emph{ergodic} Borel probability measure on $X$ is
either the Haar measure $m_X$ on $X$ or the counting probability measure on a finite $\Gamma$-orbit. 
By Proposition \ref{BQ1}, there are (at most) countably many finite $\Gamma$-orbits in $X$, and
each element in a finite $\Gamma$-orbit has finite order. 
\end{proof}

The following corollary, in combination with Corollary \ref{niceness}, proves Proposition \ref{prop4}.

\begin{corollary}
The action of $\Ad(\SL_d(\bZ))$ on $M_d^0(\bR)$ is strongly irreducible.
\end{corollary}

\begin{proof}
Let $\Gamma_o$ be a finite-index subgroup of $\Gamma = \Ad(\SL_d(\bZ))$ and let $V = M_d^0(\bR)$. 
We note that in this case, the Zariski closure $G := \overline{\Gamma_o}^{\bZ}$ equals $\PSL_d(\bR)$ 
by the Borel Density Theorem \cite{F}, which is Zariski-connected (since it is algebraically simple) and it
acts irreducibly on $V$. Indeed, any linear subspace of $M_d^0(\bR) \cong \gsl_d(\bR)$, which is invariant under the adjoint representation,  
is an ideal in $\gsl_d(\bR)$. Since $\gsl_d(\bR)$ is simple as a Lie algebra, we see that the adjoint 
action is irreducible. By Proposition \ref{BQ1}, this
shows that $\Gamma$ acts strongly irreducibly.
\end{proof}

\section{Acknowledgements}

The authors are grateful to Yv\'es Benoist for enlightening and very encouraging discussions. The second 
author would also like to thank the Department of Mathematical Sciences at Chalmers University, Gothenburg,
for their hospitality during the time this paper was written.

\end{document}